\documentclass{amsart}
\usepackage{amsfonts, amsbsy, amsmath, amssymb}
\usepackage{tikz}

\newtheorem{thm}{Theorem}[section]
\newtheorem{lem}[thm]{Lemma}
\newtheorem{cor}[thm]{Corollary}
\newtheorem{prop}[thm]{Proposition}
\newtheorem{exmp}[thm]{Example}

\newtheorem{conj}[thm]{Conjecture}

\newtheorem{thm-con}[thm]{Theorem-Conjecture}
\numberwithin{equation}{section}

\theoremstyle{definition}

\allowdisplaybreaks

\newcommand{\f}{\Bbb F}
\newcommand{\fin}{f_{\text{\rm inv}}}
\newcommand{\SF}{\text{\rm SF}_n}
\newcommand{\aut}{\text{\rm Aut}}
\newcommand{\rank}{\text{\rm rank}}

\begin{document}

\title[Two Absolutely Irreducible Polynomials]{Two Absolutely Irreducible Polynomials over $\f_2$ and Their Applications to a Conjecture by Carlet}

\author[X. Hou]{Xiang-dong Hou}
\address{Department of Mathematics and Statistics,
University of South Florida, Tampa, FL 33620}
\email{xhou@usf.edu}

\author[S. Zhao]{Shujun Zhao}
\address{Department of Mathematics and Statistics,
University of South Florida, Tampa, FL 33620}
\email{shujunz@usf.edu}



\begin{abstract}
Two polynomials $F_k(X_1,\dots,X_k)$ and $\Theta_k(X_1,\dots,X_k)$ over $\f_2$ arose from the study of a conjecture by C. Carlet about the sum-freedom of the multiplicative inverse function of $\f_{2^n}$. Both $F_k$ and $\Theta_k$ are homogeneous and symmetric with $\deg F_k=2^k-2$ and $\deg\Theta_k=2^{k-1}$. It is known that $F_k$ is absolutely irreducible for $k\ge 3$. Using the Lang-Weil bound and a curious connection between $F_k$ and $\Theta_k$, we show that $\Theta_k$ ($k\ge 3$) is also absolutely irreducible. This conclusion allows us to improve several existing results about Carlet's conjecture.
\end{abstract}

\maketitle

\section{Introduction}

A function $f:\f_{2^n}\to\f_{2^n}$ is said to be {\em $k$th order sum-free} if $\sum_{x\in A}f(x)\ne 0$ for every $k$-dimensional $\f_2$-affine subspace $A$ of $\f_{2^n}$. Sum-free functions were introduced recently by C. Carlet as a natural generalization of {\em almost perfect nonlinear} (APN) functions; the latter are precisely the 2nd order sum-free functions \cite{Carlet-CEA-2024/841, Carlet-CEA-2024/1007}. For applications of APN functions and sum-free functions in cryptography, see \cite[Chapter~11]{Carlet-2021} and \cite{Carlet-CEA-2024/841, Nyberg-LNCS-1992}. Let $\fin:\f_{2^n}\to\f_{2^n}$ be the multiplicative inverse function defined by $\fin(x)=x^{-1}$ for $x\in\f_{2^n}^*$ and $\fin(0)=0$. It is a challenging problem to determine the values of $1\le k\le n$ such that $\fin$ is $k$th order sum-free on $\f_{2^n}$. Let
\begin{equation}\label{SF}
\SF=\{1\le k\le n-1: \fin\ \text{is $k$th order sum-free}\}
\end{equation}
and 
\begin{equation}\label{Kn}
\mathcal K_n=\{1,\dots,n-1\}\setminus\SF.
\end{equation}
We know from \cite{Carlet-CEA-2024/841, Carlet-CEA-2024/1007, Carlet-Hou} that
\[
\SF=\{1,n-1\} \ \text{for even}\ n
\]
and 
\[
\SF\supset\{1,2,n-2,n-1\}\ \text{for odd}\ n\ge 3.
\]
The following conjecture, known as Carlet's conjecture, is supported by strong evidence, both theoretical and numerical.

\begin{conj}[Carlet \cite{Carlet-CEA-2024/1007}]\label{C1.1} 
For odd $n\ge 3$, $\SF=\{1,2,n-2,n-1\}$.
\end{conj}
In the next section, we will review the state of knowledge on Conjecture~\ref{C1.1} prior to the present paper. 

There are two criteria for $\fin$ not to be $k$th order sum-free \cite{Carlet-CEA-2024/1007, Carlet-Hou, EHRZ}. The two criteria are stated in terms of two polynomials $F_k(X_1,\dots,X_k)$ and $\Theta_k(X_1,\dots,X_k)$ over $\f_2$, respectively: $\fin$ on $\f_{2^n}$ is not $k$th order sum-free if and only if there exist $u_1,\dots,u_k\in\f_{2^n}$, linearly independent over $\f_2$, such that $F_k(u_1,\dots,u_k)=0$, equivalently, there exist $v_1,\dots,v_k\in\f_{2^n}$, linearly independent over $\f_2$, such that $\Theta_k(v_1,\dots,v_k)=0$. The polynomials $F_k$ and $\Theta_k$ (defined in Section~3) are constructed quite differently. It is a curious fact that they are connected through the sum-freedom of $\fin$. In Section~3, we give a full account of the two criteria and the polynomials $F_k$ and $\Theta_k$.

It is known that for $k\ge 3$, $F_k(X_1,\dots,X_k)$ is absolutely irreducible (\cite[Lemma~4.1]{Carlet-Hou}), i.e., irreducible over the algebraic closure $\overline\f_2$ of $\f_2$. It was also proved in \cite{EHRZ} by an ad hoc method that $\Theta_4(X_1,\dots,X_4)$ is absolutely irreducible. In Section~4, we prove that in general, $\Theta_k(X_1,\dots,X_k)$ is absolutely irreducible for $k\ge 3$. The proof uses the Lang-Weil bound on the number of zeros of absolutely irreducible polynomials over finite fields and the aforementioned connection between $F_k$ and $\Theta_k$.

The absolute irreducibility of $\Theta_k$, in return, allows us to strengthen several existing results concerning Conjecture~\ref{C1.1}. Specifically, we prove the following statements in Section~5:
\begin{itemize}
\item When $k\ge 3$ and $n\ge 10.8k-15.7$, $\fin$ is not $k$th order sum-free. This improves a result of \cite{Carlet-Hou} which requires $n\ge 10.8k-5$; see \cite[Remark~4.4]{Carlet-Hou}.

\medskip

\item Assume $3\le k\le n-3$. If $3\le k\le\lceil n/3\rceil+2$ or $\lfloor 2n/3\rfloor-2\le k\le n-3$, then $\fin$ is not $k$th order sum-free. The same conclusion can be derived from the results of \cite{Carlet-Hou} when $n$ is sufficiently large ($n\ge 38$). Here we removed such a requirement on $n$.

\medskip

\item Assume $3\le k\le n-3$. If $3\le k\le 12$, then $k\in\mathcal K_n$, i.e., $\fin$ is not $k$th order sum-free. Previous best result in this form was that $3,4\in\mathcal K_n$; see \cite[(update), \S 5.2.4]{Carlet-CEA-2024/1007} and \cite[Theorem~3.2]{EHRZ}.

\medskip

\item Conjecture~1.1 holds for odd $n\le 27$. It was previously known that the conjecture holds for odd $n\le 15$ (\cite{EHRZ}).

\medskip

\item Conjecture~1.1 holds when $7\mid n$, with possible exceptions $(n,k)=(49,23)$, $(49,26)$.
\end{itemize}

\section{Known Results on Carlet's Conjecture}

Recall that $\SF=\{1\le k\le n-1:\fin\ \text{is $k$th order sum-free}\}$ and $\mathcal K_n=\{1,\dots,n-1\}\setminus\SF$. The original form of Carlet's conjecture in \cite[Conclusion]{Carlet-CEA-2024/841} stated that
\[
\SF=\begin{cases}
\{1,n-1\}&\text{for even}\ n,\cr
\{1,2,n-2,n-1\}&\text{for odd}\ n\ge 3.
\end{cases}
\]
Significant progresses have been made on this conjecture in several recent works \cite{Carlet-CEA-2024/841, Carlet-CEA-2024/1007, Carlet-Hou, EHRZ}. The following is a summary of the known results in this regard. We always assume $1\le k\le n-1$.
\begin{enumerate}
\item If $\gcd(k,n)>1$, then $k\in\mathcal K_n$ (obvious).

\medskip

\item $k\in\SF$ if and only if $n-k\in\SF$ (\cite{Carlet-CEA-2024/1007}).

\medskip

\item Carlet's (original) conjecture holds in each of the following cases:
\begin{itemize}
\item $n\le 16$ (\cite{EHRZ}).

\item $n$ is divisible by $2$ or $3$ or $5$ (\cite{Carlet-Hou, EHRZ}).

\item The smallest prime divisor $l$ of $n$ satisfies $(l-1)(l+2)\le (n+1)/2$ (\cite{EHRZ}).
\end{itemize}

\medskip

\item For odd $n\ge 3$, $\SF\supset\{1,2,n-2,n-1\}$ (\cite{Nyberg-LNCS-1994}).

\medskip

\item If $1\le k,l\le n-1$ are such that $k,l\in\mathcal K_n$ and $kl<n$, then $k+l\in\mathcal K_n$ (\cite{Carlet-CEA-2024/1007}).

\medskip

\item If $X^n-1$ has a factor $X^k+a_{k-1}X^{k-1}+\cdots+a_2X^2+a_0\in\f_2[X]$, then $k\in\mathcal K_n$ (\cite{Carlet-CEA-2024/1007, Carlet-Hou}).

\medskip

\item If $n\ge 10.8k-5$, then $k\in\mathcal K_n$ (\cite{Carlet-Hou}). A more precise form of this result is also given in \cite{Carlet-Hou}.

\medskip

\item Assume $3\le k\le n-3$. If $k=3,4$, then $k\in\mathcal K_n$ (\cite{Carlet-CEA-2024/1007, EHRZ}).

\medskip

\item Assume $n\ge 38$ and $3\le k\le n-3$. If $3\le k\le\lceil n/3\rceil+2$ or $\lfloor 2n/3\rfloor-2\le k\le n-3$, then $k\in\mathcal K_n$ (\cite{Carlet-Hou}).
\end{enumerate}

Since Carlet's original conjecture has been confirmed for even $n$, it is reformulated as Conjecture~\ref{C1.1}. In Section~5, the above results (3), (7) -- (9) will be strengthened.

\section{Two Criteria}

Throughout the paper, affine subspaces and subspaces of $\f_{2^n}$ are meant to be $\f_2$-affine subspaces and $\f_2$-subspaces of $\f_{2^n}$. It is known that for all affine subspaces $A$ of $\f_{2^n}$ not containing $0$, $\sum_{x\in A}1/x\ne 0$ (\cite{Carlet-CEA-2024/841}). Therefore, $\fin$ is not $k$th order sum-free if and only if $\sum_{0\ne u\in A}1/u=0$ for some $k$-dimensional subspace $E$ of $\f_{2^n}$. We call a subspace $E$ of $\f_{2^n}$ a {\em zero-sum subspace} if 
\[
\sum_{0\ne u\in E}\frac 1u=0.
\]
There are two criteria for a subspace of $\f_{2^n}$ to be sum-free. Both criteria were originally due to Carlet \cite{Carlet-CEA-2024/841, Carlet-CEA-2024/1007} and were reformulated later in \cite{Carlet-Hou, EHRZ}. The two criteria are stated in terms of two polynomials $F_k$ and $\Theta_k$, respectively. We first recall the constructions of $F_k$ and $\Theta_k$.

\subsection{Constructions of $F_k$ and $\Theta_k$}\

Let
\[
\Delta(X_1,\dots,X_k)=\left|
\begin{matrix}
X_1&\cdots&X_k\cr
X_1^2&\cdots&X_k^2\cr
\vdots&&\vdots\cr
X_1^{2^{k-1}}&\cdots&X_k^{2^{k-1}}
\end{matrix}\right|\in\f_2[X_1,\dots,X_k]
\]
and
\[
\Delta_1(X_1,\dots,X_k)=\left|
\begin{matrix}
X_1&\cdots&X_k\cr
X_1^{2^2}&\cdots&X_k^{2^2}\cr
\vdots&&\vdots\cr
X_1^{2^k}&\cdots&X_k^{2^k}
\end{matrix}\right|\in\f_2[X_1,\dots,X_k].
\]
$\Delta(X_1,\dots,X_k)$ is the {\em Moore determinant} over $\f_2$ (\cite{Moore-BAMS-1896}), and it is well known that $u_1\dots,u_k\in\overline\f_2$ (the algebraic closure of $\f_2$) are linearly independent if and only $\Delta(u_1,\dots,u_k)\ne 0$. It is also known that $\Delta(X_1,\dots,X_k)$ divides $\Delta_1(X_1,\dots,X_k)$ (\cite[\S4]{Carlet-Hou}). Define
\[
F_k(X_1,\dots,X_k)=\frac{\Delta_1(X_1,\dots,X_k)}{\Delta(X_1,\dots,X_k)}.
\]

A partition of a positive integer is called a {\em $2$-adic partition} if all its parts are powers of $2$. Let $\Lambda_k$ denote the set of all $2$-adic partitions of $2^{k-1}$ with at most $k$ parts. Define
\[
\Theta_k(X_1,\dots,X_k)=\sum_{\lambda\in\Lambda_k}m_\lambda(X_1,\dots,X_k),
\]
where $m_\lambda(X_1,\dots,X_k)$ is the monomial symmetric polynomial associated to the partition $\lambda$ (\cite[I.2]{Macdonald-1995}). For example, when $k=4$,
\[
\Lambda_4=\{(2^3), (2^2,2^2) (2^2,2,2),(2^2,2,1,1),(2,2,2,2)\}
\]
and
\[
\Theta_4(X_1,\dots,X_4)=m_{(2^3)}+m_{(2^2,2^2)}+m_{(2^2,2,2)}+m_{(2^2,2,1,1)}+m_{(2,2,2,2)},
\]
where
\[
m_{(2^3)}=X_1^{2^3}+X_2^{2^3}+X_3^{2^3}+X_4^{2^3},
\]
\[
m_{(2^2,2^2)}=X_1^{2^2}X_2^{2^2}+X_1^{2^2}X_3^{2^2}+X_1^{2^2}X_4^{2^2}+X_2^{2^2}X_3^{2^2}+X_2^{2^2}X_4^{2^2}+X_3^{2^2}X_4^{2^2},
\]
\begin{align*}
m_{(2^2,2,2)}=\,&X_2^{2^2}X_3^2X_4^2+X_2^2X_3^{2^2}X_4^2+X_2^2X_3^2X_4^{2^2}\cr
&+X_1^{2^2}X_3^2X_4^2+X_1^2X_3^{2^2}X_4^2+X_1^2X_3^2X_4^{2^2}\cr
&+X_1^{2^2}X_2^2X_4^2+X_1^2X_2^{2^2}X_4^2+X_1^2X_2^2X_4^{2^2}\cr
&+X_1^{2^2}X_2^2X_3^2+X_1^2X_2^{2^2}X_3^2+X_1^2X_2^2X_3^{2^2}.
\end{align*}
\begin{align*}
m_{(2^2,2,1,1)}=\,&X_1^{2^2}X_2^2X_3X_4+X_1^{2^2}X_2X_3^2X_4+X_1^{2^2}X_2X_3X_4^2\cr
&+X_1^2X_2^{2^2}X_3X_4+X_1X_2^{2^2}X_3^2X_4+X_1X_2^{2^2}X_3X_4^2\cr
&+X_1^2X_2X_3^{2^2}X_4+X_1X_2^2X_3^{2^2}X_4+X_1X_2X_3^{2^2}X_4^2\cr
&+X_1^2X_2X_3X_4^{2^2}+X_1X_2^2X_3X_4^{2^2}+X_1X_2X_3^2X_4^{2^2},
\end{align*}
\[
m_{(2,2,2,2)}=X_1^2X_2^2X_3^2X_4^2.
\]
The polynomial $\Theta_k(X_1,\dots,X_k)$ has integer coefficients and is treated as a polynomial over $\f_2$.

To summarize, both $F_k(X_1\dots,X_k)$ and $\Theta_k(X_1,\dots,X_k)$ are homogeneous and symmetric with $\deg F_k=2^k-2$ and $\deg\Theta_k=2^{k-1}$.

\subsection{The Criteria}\

Let $E$ be a $k$-dimensional subspace of $\f_{2^n}$, and define $L_E(X)=\prod_{u\in E}(X-u)\in\f_{2^n}[X]$. Then $L_E:\f_{2^n}\to \f_{2^n}$ is an $\f_2$-linear map with $\ker L_E=E$ (\cite[Theorem~3.52]{Lidl-Niederreiter-FF-1997}), hence $L_E(\f_{2^n})$ is an $(n-k)$-dimensional subspace of $\f_{2^n}$. Denote $L_E(\f_{2^n})$ by $E'$. It is known that $E''=E$, where $E''=L_{E'}(\f_{2^n})$; see \cite[Theorem~11.35]{Berlekamp-1968}. We also define $E^\bot=\{x\in\f_{2^n}:\text{Tr}_{2^n/2}(xy)=0\ \text{for all}\ y\in E\}$. Then the two criteria can be stated as follows:

\begin{thm}[Criterion~1 {\cite{Carlet-CEA-2024/841}, \cite[\S 4]{Carlet-Hou}}]\label{cri1}
A $k$-dimensional subspace $E$ of $\f_{2^n}$ is a zero-sum subspace if and only if 
\[
F_k(u_1,\dots,u_k)=0,
\]
where $u_1,\dots,u_k$ is any basis of $E$.
\end{thm}

\begin{thm}[Criterion~2 {\cite[\S 6.7]{Carlet-CEA-2024/1007}, \cite[Theorem~3.11]{EHRZ}}]\label{cri2}
A $k$-dimensional subspace $E$ of $\f_{2^n}$ is a zero-sum subspace if and only if 
\[
\Theta_k(v_1,\dots,v_k)=0,
\]
where $v_1,\dots,v_k$ is any basis of $(E')^\bot$.
\end{thm}

\begin{exmp}\label{E3.3}\rm
Recall that $\Lambda_k$ is the set of all $2$-adic partitions of $2^{k-1}$ with at most $k$ parts. For $k=5$ we have
\begin{align*}
\Lambda_5=\{&(2^4), (2^3,2^3), (2^3,2^2,2^2), (2^3,2^2,2,2), (2^3,2^2,2,1,1),\cr &(2^3,2,2,2,2), (2^2,2^2,2^2,2^2),(2^2,2^2,2^2,2,2)\}
\end{align*}
and
\[
\Theta_5(X_1,\dots,X_5)=\sum_{\lambda\in\Lambda_5}m_\lambda(X_1,\dots,X_5).
\]
$\Theta_5$ is a polynomial of degree $2^4$ with 185 terms, which can be easily generated by computer.

Let $\f_{2^{17}}=\f_2[X]/(f)$, where $f=X^{17} + X^3 + 1\in\f_2[X]$ is irreducible.
Through a computer search, we found a $5$-dimensional zero-sum subspace $E\subset\f_{2^{17}}$ with a basis $u_1,\dots,u_5$, where
\[
\left[\begin{matrix} u_1\cr \vdots\cr u_5\end{matrix}\right]=
\left[
\begin{array}{ccccccccccccccccc}
 1 & 0 & 0 & 0 & 0 & 0 & 0 & 0 & 0 & 0 & 0 & 0 & 0 & 0 & 0 & 0 & 0 \\
 0 & 1 & 0 & 0 & 0 & 0 & 0 & 1 & 1 & 0 & 0 & 1 & 1 & 0 & 1 & 0 & 0 \\
 0 & 0 & 1 & 0 & 0 & 1 & 1 & 1 & 0 & 0 & 0 & 1 & 0 & 0 & 1 & 1 & 0 \\
 0 & 0 & 0 & 1 & 0 & 1 & 0 & 1 & 0 & 0 & 1 & 0 & 0 & 1 & 0 & 1 & 1 \\
 0 & 0 & 0 & 0 & 0 & 0 & 0 & 0 & 0 & 1 & 0 & 1 & 1 & 0 & 0 & 1 & 1 \\
\end{array}
\right]
\left[\begin{matrix} X^0\cr \vdots\cr X^{16}\end{matrix}\right],
\]
that is, $F_5(u_1,\dots,u_5)=0$. The subspace $(E')^\bot$ has a basis $v_1,\dots,v_5$, where
\[
\left[\begin{matrix} v_1\cr \vdots\cr v_5\end{matrix}\right]=
\left[
\begin{array}{ccccccccccccccccc}
 1 & 0 & 1 & 0 & 0 & 1 & 0 & 0 & 0 & 0 & 1 & 0 & 0 & 1 & 1 & 1 & 0 \\
 0 & 1 & 0 & 0 & 0 & 1 & 0 & 0 & 0 & 0 & 0 & 0 & 0 & 0 & 0 & 1 & 0 \\
 0 & 0 & 0 & 1 & 0 & 0 & 0 & 0 & 0 & 1 & 1 & 1 & 1 & 1 & 0 & 0 & 0 \\
 0 & 0 & 0 & 0 & 1 & 0 & 0 & 0 & 1 & 0 & 1 & 1 & 0 & 1 & 1 & 0 & 0 \\
 0 & 0 & 0 & 0 & 0 & 0 & 0 & 1 & 1 & 1 & 0 & 1 & 0 & 0 & 1 & 0 & 0 \\
\end{array}
\right]
\left[\begin{matrix} X^0\cr \vdots\cr X^{16}\end{matrix}\right].
\]
We verified that $\Theta_5(v_1,\dots,v_5)=0$.
\end{exmp}

\begin{exmp}\label{E3.4}\rm
Let $\f_{2^{19}}=\f_2[X]/(f)$, where $f=X^{19} + X^5 + X^2 + X + 1\in\f_2[X]$ is irreducible. Again, through a computer search, we found a $5$-dimensional zero-sum subspace $E\subset\f_{2^{19}}$ with a basis $u_1,\dots,u_5$, where
\[
\left[\begin{matrix} u_1\cr \vdots\cr u_5\end{matrix}\right]\!=\!
\left[
\begin{array}{ccccccccccccccccccc}
 1 & 0 & 0 & 0 & 0 & 0 & 0 & 0 & 0 & 0 & 0 & 0 & 0 & 0 & 0 & 0 & 0 & 0 & 0 \\
 0 & 1 & 0 & 0 & 0 & 0 & 0 & 0 & 0 & 1 & 0 & 1 & 0 & 0 & 0 & 1 & 1 & 1 & 0 \\
 0 & 0 & 1 & 0 & 0 & 0 & 1 & 0 & 0 & 1 & 1 & 1 & 0 & 0 & 1 & 0 & 0 & 0 & 0 \\
 0 & 0 & 0 & 1 & 0 & 0 & 0 & 0 & 0 & 0 & 1 & 0 & 1 & 0 & 1 & 0 & 0 & 0 & 1 \\
 0 & 0 & 0 & 0 & 0 & 1 & 0 & 1 & 0 & 0 & 1 & 0 & 0 & 0 & 1 & 0 & 0 & 1 & 1 \\
\end{array}
\right]\!
\left[\begin{matrix} X^0\cr \vdots\cr X^{18}\end{matrix}\right],
\]
that is, $F_5(u_1,\dots,u_5)=0$. The subspace $(E')^\bot$ has a basis $v_1,\dots,v_5$, where
\[
\left[\begin{matrix} v_1\cr \vdots\cr v_5\end{matrix}\right]\!=\!
\left[
\begin{array}{ccccccccccccccccccc}
 1 & 0 & 0 & 0 & 0 & 0 & 0 & 1 & 1 & 1 & 1 & 1 & 0 & 0 & 0 & 1 & 1 & 0 & 0 \\
 0 & 1 & 0 & 0 & 0 & 0 & 0 & 1 & 1 & 1 & 1 & 1 & 1 & 0 & 1 & 1 & 0 & 1 & 0 \\
 0 & 0 & 1 & 0 & 0 & 0 & 0 & 1 & 0 & 1 & 1 & 1 & 0 & 1 & 0 & 1 & 1 & 0 & 0 \\
 0 & 0 & 0 & 1 & 0 & 0 & 1 & 0 & 1 & 1 & 1 & 1 & 0 & 1 & 1 & 0 & 1 & 1 & 1 \\
 0 & 0 & 0 & 0 & 0 & 1 & 1 & 0 & 0 & 1 & 0 & 1 & 1 & 0 & 0 & 0 & 1 & 0 & 0 \\
\end{array}
\right]\!
\left[\begin{matrix} X^0\cr \vdots\cr X^{18}\end{matrix}\right].
\]
We verified that $\Theta_5(v_1,\dots,v_5)=0$.
\end{exmp}

\subsection{Observations and Generalizations}\

In general, for an $\f_q$-subspace $E$ of $\f_{q^n}$, define $L_E(X)=\prod_{u\in E}(X-u)\in\f_{q^n}[X]$. Then $L_E(X)$ is a $q$-polynomial over $\f_{q^n}$ (\cite[Theorem~3.52]{Lidl-Niederreiter-FF-1997}). Therefore, $L_E:\f_{q^n}\to\f_{q^n}$ is an $\f_q$-linear map with $\ker L_E=E$, and $E':=L_E(\f_{q^n})$ is an $(n-k)$-dimensional $\f_q$-subspace of $\f_{q^n}$. It was proved in \cite[Theorem~11.35]{Berlekamp-1968} that $E''=E$, i.e., $L_{E'}(\f_{q^n})=E$. We wish to point out that this fact can be seen from a ring-theoretic point of view. For this purpose, we include a proof which is the same as the one in \cite{Berlekamp-1968} but is stated in the language of ring theory.

\begin{lem}\label{L3.5}
Let $R$ be a ring without zero divisors. If $a,b\in R$ are such that $ab$ is in the center of $R$, then $ab=ba$.
\end{lem}

\begin{proof}
We may assume that $b\ne 0$. We have $abb=bab$, whence $ab=ba$.
\end{proof}

\begin{proof}[Proof that $E''=E$, where $E$ is an $\f_q$-subspace of $\f_{q^n}$]
Let $\mathcal L(q,n)$ denote the ring of $q$-polynomials over $\f_{q^n}$ whose addition is the ordinary and whose multiplication is composition. $\mathcal L(q,n)$ is isomorphic to the shew polynomial ring $\f_{q^n}[X;\sigma]$, where $\sigma$ is the Frobenius of $\f_{q^n}$ over $\f_q$; see \cite[Theorem~2.34]{Hou-ams-gsm-2018}. The center of $\mathcal L(q,n)$ consists of $q^n$-polynomials over $\f_q$. Since $(L_{E'}\circ L_E)(x)=0$ for all $x\in\f_{q^n}$ (by definition) and $L_{E'}\circ L_E$ is monic of degree $q^n$, we have $L_{E'}\circ L_E=X^{q^n}-X$, which belongs to the center of $\mathcal L(q,n)$. By Lemma~\ref{L3.5}, $L_{E'}\circ L_E=L_E\circ L_{E'}$. Now, $L_{E''}\circ L_{E'}=X^{q^n}-X=L_E\circ L_{E'}$, whence $L_{E''}=L_E$, i.e., $E''=E$.
\end{proof}

Let $\mathcal V_k(\f_{q^n})$ denote the set of all $k$-dimensional $\f_q$-subspaces of $\f_{q^n}$. Define
\[
\begin{array}{cccc}
\alpha_k:&\mathcal V_k(\f_{q^n})&\longrightarrow &\mathcal V_{n-k}(\f_{q^n})\cr
&E&\longmapsto & L_E(\f_{q^n}),
\end{array}
\]
and 
\[
\begin{array}{cccc}
\beta_k:&\mathcal V_k(\f_{q^n})&\longrightarrow &\mathcal V_{n-k}(\f_{q^n})\cr
&E&\longmapsto & E^\bot,
\end{array}
\]
where $E^\bot=\{x\in\f_{q^n}:\text{Tr}_{q^n/q}(xy)=0\ \text{for all}\ y\in E\}$. Then both $\alpha_k$ and $\beta_k$ are bijections with $\alpha_{n-k}\circ\alpha_k=\text{id}$ and $\beta_{n-k}\circ\beta_k=\text{id}$. Therefore, $\gamma_k:=\beta_{n-k}\circ\alpha_k:\mathcal V_k(\f_{q^n})\to\mathcal V_k(\f_{q^n})$ is a bijection with $\gamma_k^{-1}=\alpha_{n-k}\circ\beta_k$.


Every $\f_q$-linear map from $\f_{q^n}$ to $\f_{q^n}$ can be uniquely represented by a $q$-polynomial over $\f_{q^n}$ modulo $X^{q^n}-X$. Given a $k$-dimensional $\f_q$-subspace $E$ of $\f_{q^n}$, there are $(q^n-q^0)(q^n-q^1)\cdots(q^n-q^{n-k-1})$ $\f_q$-linear maps $f:\f_{q^n}\to\f_{q^n}$ such that $\ker f=E$. Among these $\f_q$-linear maps, treated as $q$-polynomials over $\f_{q^n}$, $L_E(X)$ is the unique monic one with degree $\le |E|$. $L_E(X)$ is canonical in the sense that its construction does not involve a basis of $E$ or a basis of $\f_{q^n}$ over $\f_q$.

On the other hand, it is known that given a $q$-polynomial $L=\sum_{i=0}^ka_iX^{q^i}$ of degree $\le q^k$ over $\f_{q^n}$, $\ker L\in\mathcal V_k(F_{q^n})$ if and only if $a_k\ne 0$ and $a_0,\dots,a_k$ satisfy the equation
\begin{equation}\label{C-q}
C^{(q^0)}C^{(q^1)}\cdots C^{(q^{n-1})}=I_k,
\end{equation}
where 
\[
C=\left[
\begin{matrix}&&&-a_0/a_k\cr
1&&&-a_1/a_k\cr
&\ddots&&\vdots\cr
&&1&-a_{k-1}/a_k
\end{matrix}
\right],
\]
$C^{(q^i)}$ is the result after applying $(\ )^{q^i}$ to the entries of $C$, and $I_k$ is the $k\times k$ identity matrix; see \cite{Csajbok-Marino-Polverino-Zullo-FFA-2019, McGuire-Sheekey-FFA-2019}. Therefore, the elements of $\mathcal V_k(\f_{q^n})$ are in one-to-one correspondence with the points $[a_0:a_1:\cdots:a_k]$ in the projective space $\text{PG}(k,\f_{q^n})$ with $a_k\ne0$ whose homogeneous coordinates $a_0/a_k,\dots,a_{k-1}/a_k$ satisfy \eqref{C-q}. It is not difficult to see that under this correspondence, the map $\gamma_k:\mathcal V_k(\f_{q^n})\to V_k(\f_{q^n})$ maps $[a_0:a_1:\dots:a_k]$ to $[a_k^{q^0}:a_{k-1}^{q^1}:\cdots:a_0^{q^k}]$.

\medskip

Let's consider a more general situation. Let $K/F$ be a finite cyclic extension (a Galois extension with a finite cyclic Galois group) with $\aut(K/F)=\langle\sigma\rangle$. Then $\sigma^0,\dots,\sigma^{n-1}$, where $n=[K:F]$, is a $K$-basis of $\text{Hom}_F(K,K)$. 

\begin{prop}\label{P3.7}
In the above notation, for each $k$-dimensional $F$-subspace $E$ of $K$, there is a unique tuple $(a_0,\dots,a_k)\in K^{k+1}$ with $a_k=1$ such that $\ker(a_0\sigma^0+\cdots+a_k\sigma^k)=E$.
\end{prop}

\begin{lem}[{\cite[Theorem~5]{Gow-Quinlan-LAA-2009}}] \label{dim}
In the above notation, for $(a_0,\dots,a_k)\in K^{k+1}$ with $a_k\ne 0$, we have
\[
\dim_F\ker(a_0\sigma^0+\cdots+a_k\sigma^k)\le k.
\]
\end{lem}

\begin{lem}\label{L3.8}
Let $K/F$ be a finite dimensional Galois extension with $\aut(K/F)=\{\sigma_1,\dots,\sigma_n\}$, $n=[K:F]$. Then $k$ elements $b_1,\dots,b_k\in K$ are linearly independent over $F$ if and only the $n\times k$ matrix
\begin{equation}\label{B}
B=\left[
\begin{matrix}
\sigma_1(b_1)&\cdots&\sigma_1(b_k)\cr
\vdots&&\vdots\cr
\sigma_n(b_1)&\cdots&\sigma_n(b_k)
\end{matrix}\right]
\end{equation}
has rank $k$.
\end{lem}

\begin{proof}
When $F$ is finite, this fact appeared as Lemma~2.30 in \cite{Hou-ams-gsm-2018}. The proof of the general case is identical.

\medskip
($\Leftarrow$) Assume to the contrary that $\sum_{i=1}^k\beta_ib_i=0$ for some $0\ne(\beta_1,\dots,\beta_k)\in F^k$. Then $B(\beta_1,\dots,\beta_k)^T=0$, which is a contradiction.

\medskip
($\Rightarrow$) First assume $k=n$. Then the $(i,j)$-entry of $B^TB$ is 
\[
\sum_{l=1}^n\sigma_l(b_ib_j)=\text{Tr}_{K/F}(b_ib_j).
\]
Since $(x,y)\mapsto\text{Tr}_{K/F}(xy)$ is a nondegenerate $F$-bilinear form on $K$, we have $\det(B^TB)\ne 0$, and hence $B$ is nonsingular.

Now assume $1\le k\le n$. Extend $b_1,\dots,b_k$ to an $F$-basis $b_1,\dots,b_n$ of $K$. Then the columns of 
\[
\left[
\begin{matrix}
\sigma_1(b_1)&\cdots&\sigma_1(b_n)\cr
\vdots&&\vdots\cr
\sigma_n(b_1)&\cdots&\sigma_n(b_n)
\end{matrix}\right]
\]
are linearly independent over $K$, whence $\rank\, B=k$.
\end{proof}

\begin{proof}[Proof of Proposition~\ref{P3.7}]
Let $u_1,\dots,u_k$ be a basis of $E$ over $F$, and write
\[
L=\left|
\begin{matrix}
\sigma^0(u_1)&\cdots&\sigma^0(u_k)&\sigma^0\cr
\vdots&&\vdots&\vdots\cr
\sigma^k(u_1)&\cdots&\sigma^k(u_k)&\sigma^k
\end{matrix}\right|=b_0\sigma^0+\cdots+b_k\sigma^k,
\]
where 
\[
b_k=\left|
\begin{matrix}
\sigma^0(u_1)&\cdots&\sigma^0(u_k)\cr
\vdots&&\vdots\cr
\sigma^{k-1}(u_1)&\cdots&\sigma^{k-1}(u_k)
\end{matrix}\right|.
\]
By Lemma~\ref{L3.8}, $L\ne 0$. Clearly, $E\subset \ker(L)$. Then it follows from Lemma~\ref{dim} that $b_k\ne 0$ and $\ker(L)=E$. Therefore, $(a_0,\dots,a_k)=(b_0/b_k,\dots,b_k/b_k)$ has the desired property.

For the uniqueness of $(a_0,\dots,a_k)$, assume that $(a_0',\dots,a_k')\in K^{k+1}$ is such that $a_k'\ne 0$ and $\ker(a_0'\sigma^0+\cdots+a_k'\sigma^k)=E$. Then 
\[
E\subset\ker((a_0-a_0')\sigma^0+\cdots+(a_{k-1}-a_{k-1}')\sigma^{k-1}),
\]
and Lemma~\ref{dim} forces $a_i=a_i'$ for all $1\le i\le k-1$.
\end{proof}

We denote the unique $F$-linear map $a_0\sigma^0+\cdots+a_k\sigma^k$ in Proposition~\ref{P3.7} by $L_E$. It is clear from the above proof that
\[
L_E=\left|
\begin{matrix}
\sigma^0(u_1)&\cdots&\sigma^0(u_k)\cr
\vdots&&\vdots\cr
\sigma^{k-1}(u_1)&\cdots&\sigma^{k-1}(u_k)
\end{matrix}\right|^{-1}
\left|
\begin{matrix}
\sigma^0(u_1)&\cdots&\sigma^0(u_k)&\sigma^0\cr
\vdots&&\vdots&\vdots\cr
\sigma^k(u_1)&\cdots&\sigma^k(u_k)&\sigma^k
\end{matrix}\right|,
\]
where $u_1,\dots,u_k$ is a basis of $E$ over $F$, and $L_E$ does not depend on the choice of the basis.

Let $E'=L_E(K)$, which is an $F$-subspace of $K$ of dimension $n-k$. We also have $E''=E$, and the proof is almost identical to the case of finite fields. Consider the skew polynomial ring $K[X;\sigma]$; its center is $F[X^n]$. Every $F$-linear map from $K$ to $K$ is of the form $f(\sigma)$, where $f\in K[X;\sigma]$ is unique modulo $X^n-1$. Write $L_E=f_E(\sigma)$, where $f_E\in K[X;\sigma]$ is monic of degree $\dim_FE$. By definition, $L_{E'}\circ L_E=0$, where $L_{E'}\circ L_E=(f_{E'}f_E)(\sigma)$. Since $f_{E'}f_E$ is monic of degree $n$, we have $f_{E'}f_E=X^n-1$, which is in the center of $K[X;\sigma]$. Therefore, by Lemma~\ref{L3.5}, $f_Ef_{E'}=f_{E'}f_E=X^n-1=f_{E''}f_{E'}$, and hence $f_E=f_{E''}$, i.e., $E=E''$.

Let $\mathcal V_k(K)$ be the set of $k$-dimensional $F$-subspaces of $K$. We also have two bijections $\alpha_k,\beta_k:\mathcal V_k(K)\to\mathcal V_{n-k}(K)$ defined by $\alpha_k(E)=L_E(K)$ and $\beta_k(E)=E^\bot$, where $(\ )^\bot$ is defined by the $F$-bilinear form $\langle x,y\rangle=\text{Tr}_{K/F}(xy)$ on $K$. Moreover, $\gamma_k=\beta_{n-k}\circ\alpha_k:\mathcal V_k(K)\to\mathcal V_k(K)$ is a bijection with $\gamma_k^{-1}=\alpha_{n-k}\circ\beta_k$. 

For $(a_0,\dots,a_k)\in K^{k+1}$, $\ker (a_0\sigma^0+\cdots+a_k\sigma^k)\in\mathcal V_k(K)$ if and only if $a_k\ne 0$ and $a_0,\dots,a_k$ satisfy the equation
\begin{equation}\label{C-sigma}
C^{\sigma^0}C^{\sigma^1}\cdots C^{\sigma^{n-1}}=I_k,
\end{equation}
where 
\[
C=\left[
\begin{matrix}&&&-a_0/a_k\cr
1&&&-a_1/a_k\cr
&\ddots&&\vdots\cr
&&1&-a_{k-1}/a_k
\end{matrix}
\right]
\]
and $C^{\sigma^i}$ is the result after applying $\sigma^i$ to the entries of $C$. Therefore, the elements of $\mathcal V_k(K)$ are in one-to-one correspondence with the points $[a_0:a_1:\cdots:a_k]\in\text{PG}(k,K)$ with $a_k\ne0$ whose homogeneous coordinates $a_0/a_k,\dots,a_{k-1}/a_k$ satisfy \eqref{C-sigma}. Under this correspondence, the map $\gamma_k:\mathcal V_k(K)\to V_k(K)$ maps $[a_0:a_1:\dots:a_k]$ to $[\sigma^0(a_k):\sigma^1(a_{k-1}):\cdots:\sigma^k(a_0)]$.

The proofs of the above claims are identical to the case of finite fields with $F=\f_q$, $K=\f_{q^n}$ and $\sigma=(\ )^q$.

\section{Absolute Irreducibility of $F_k$ and $\Theta_k$}

It has been proved that the polynomial $F_k(X_1,\dots, X_k)$ is absolutely irreducible for $k\ge 3$ (\cite[Lemma~4.1]{Carlet-Hou}). In this section, we show that the polynomial $\Theta_k(X_1,\dots,X_k)$ is also absolutely irreducible for $k\ge 3$. When $k=4$, the absolute irreducibility of $\Theta_4$ was established in \cite{EHRZ} by a rather technical method which does not seem to apply to the general situation. Our approach here is indirect and is based on the relation between $F_k$ and $\Theta_k$ described in Theorems~\ref{cri1} and \ref{cri2}.

For a field $\f$ and a polynomial $f\in\f[X_1,\dots,X_k]$, we define
\[
V_{\f^k}(f)=\{(x_1,\dots,x_k)\in\f^k:f(x_1,\dots,x_k)=0\}.
\]
Recall that a zero-sum subspace of $\f_{2^n}$ is a subspace $E\subset\f_{2^n}$ such that $\sum_{0\ne u\in E}1/u=0$. Let $Z_k$ denote the number of $k$-dimensional zero-sum subspaces of $\f_{2^n}$. By Theorems~\ref{cri1} and \ref{cri2},
\begin{align}\label{VFk}
&|V_{\f_{2^n}^k}(F_k)\setminus V_{\f_{2^n}^k}(\Delta(X_1,\dots,X_k))|=Z_k|\text{GL}(k,\f_2)|\\
&=|V_{\f_{2^n}^k}(\Theta_k)\setminus V_{\f_{2^n}^k}(\Delta(X_1,\dots,X_k))|.\nonumber
\end{align}

\begin{lem}\label{L4.1}
Assume that $f(X_1,\dots,X_k)\in\f_q[X_1,\dots,X_k]$ has $l$ distinct absolutely irreducible factors, and assume that all these irreducible factors belong to $\f_{q^m}[X_1,\dots,X_k]$. Further assume that $\gcd(f,\Delta(X_1,\dots,X_k))=1$. Then
\[
|V_{\f_{q^{ms}}^k}(f)\setminus V_{\f_{q^{ms}}^k}(\Delta(X_1,\dots,X_k))|=lq^{ms(k-1)}+O(q^{ms(k-3/2)})\quad \text{as}\ s\to\infty.
\]
\end{lem}

\begin{proof}
Write $f=f_1^{e_1}\cdots f_l^{e_l}$, where $f_1,\dots,f_l\in\f_{q^m}[X_1,\dots,X_k]$ are distinct and absolutely irreducible and $e_i>0$, $1\le i\le l$. We may assume that $e_i=1$ for all $i$ since the multiplicities $e_i$ do not affect the number of zeros of $f$ in $\f_{q^{ms}}^k$. By the Lang-Weil bound, as stated in \cite[Theorem~5.2]{Cafure-Matera-FFA-2006},
\begin{equation}\label{fi}
|V_{\f_{q^{ms}}^k}(f_i)|=q^{ms(k-1)}+O(q^{ms(k-3/2)})\quad \text{as}\ s\to\infty.
\end{equation}
By \cite[Lemma~2.2]{Cafure-Matera-FFA-2006}, for $1\le i<j\le l$,
\begin{equation}\label{fifj}
|V_{\f_{q^{ms}}^k}(f_i)\cap V_{\f_{q^{ms}}^k}(f_j)|=O(q^{ms(k-2)})\quad \text{as}\ s\to\infty.
\end{equation}
Therefore,
\begin{align*}
&\sum_{i=1}^l |V_{\f_{q^{ms}}^k}(f_i)|\cr
\ge\,&|V_{\f_{q^{ms}}^k}(f_1)\cup\cdots\cup V_{\f_{q^{ms}}^k}(f_l)|=|V_{\f_{q^{ms}}^k}(f)|\cr
\ge\,&\sum_{i=1}^l |V_{\f_{q^{ms}}^k}(f_i)|-\sum_{1\le i<j\le l}|V_{\f_{q^{ms}}^k}(f_i)\cap V_{\f_{q^{ms}}^k}(f_j)|.
\end{align*}
Thus it follows from \eqref{fi} and \eqref{fifj} that
\begin{equation}\label{f}
|V_{\f_{q^{ms}}^k}(f)|=lq^{ms(k-1)}+O(q^{ms(k-3/2)}).
\end{equation}
Also by \cite[Lemma~2.2]{Cafure-Matera-FFA-2006}, 
\begin{equation}\label{f-delta}
|V_{\f_{q^{ms}}^k}(f)\cap V_{\f_{q^{ms}}^k}(\Delta)|=O(q^{ms(k-2)}).
\end{equation}
It follows from \eqref{f} and \eqref{f-delta} that
\begin{align*}
|V_{\f_{q^{ms}}^k}(f)\setminus V_{\f_{q^{ms}}^k}(\Delta)|\,&=|V_{\f_{q^{ms}}^k}(f)|-|V_{\f_{q^{ms}}^k}(f)\cap V_{\f_{q^{ms}}^k}(\Delta)|\cr
&=lq^{ms(k-1)}+O(q^{ms(k-3/2)}).
\end{align*}
\end{proof}

\begin{lem}\label{L4.2}
When $k\ge 2$, $\gcd(\Theta_k,\Delta(X_1,\dots,X_k))=1$.
\end{lem}

\begin{proof}
Recall that 
\[
\Theta_k(X_1,\dots,X_k)=\sum_{\lambda\in\Lambda_k}m_\lambda(X_1,\dots,X_k),
\]
where $\Lambda_k$ is the set of all $2$-adic partitions of $2^{k-1}$ with at most $k$ parts. By \cite[Lemma~3.51]{Lidl-Niederreiter-FF-1997},
\[
\Delta(X_1,\dots,X_k)=\prod_{0\ne(a_1,\dots,a_k)\in\f_2^k}(a_1X_1+\cdots+a_kX_k).
\]
Assume to the contrary that $\gcd(\Theta_k,\Delta)\ne 1$. Then there exists $0\ne(a_1,\dots,a_k)\in\f_2^k$ such that $(a_1X_1+\cdots+a_kX_k)\mid\Theta_k$.

\medskip
{\bf Case 1.} Assume that only one of $a_1,\dots,a_k$ is nonzero, say $(a_1,\dots,a_k)=(0,\cdots,0,1)$. Then the value of $a_1X_1+\cdots+a_kX_k$ at $(1,0,\dots,0)$ is $0$, whence $\Theta_k(1,0,\dots,0)=0$. However,
\[
m_\lambda(1,0,\dots,0)=\begin{cases}
1&\text{if}\ \lambda=(2^{k-1}),\cr
0&\text{if $\lambda$ has more than one part}.
\end{cases}
\]
Hence $\Theta_k(1,0,\dots,0)=1$, which is a contradiction.

\medskip
{\bf Case 2.} Assume that at least two of $a_1,\dots,a_k$ are nonzero, say $(a_1,\dots,a_k)=(1,1,*,\dots,*)$. Then the value of $a_1X_1+\cdots+a_kX_k$ at $(1,1,0,\dots,0)$ is $0$, whence $\Theta_k(1,1,0,\dots,0)=0$. Note that
\[
m_\lambda(1,1,0,\dots,0)=\begin{cases}
0&\text{if}\ \lambda=(2^{k-1}),\cr
1&\text{if}\ \lambda=(2^{k-2},2^{k-2}),\cr
0&\text{if $\lambda$ has at more that 2 parts}.
\end{cases}
\]
Therefore, $\Theta_k(1,1,0,\dots,0)=1$, which is a contradiction.
\end{proof}

\begin{thm}\label{T4.3}
When $k\ge 3$, $\Theta_k(X_1,\dots,X_k)$ is absolutely irreducible.
\end{thm}

\begin{proof}
Let $l$ denote the number of distinct absolutely irreducible factors of $\Theta_k$, and assume that these factors all belong to $\f_{2^m}[X_1,\dots,X_k]$. By Lemmas~\ref{L4.1} and \ref{L4.2},
\begin{equation}\label{theta}
|V_{\f_{2^{ms}}^k}(\Theta_k)\setminus V_{\f_{2^{ms}}^k}(\Delta)|=l\,2^{ms(k-1)}+O(2^{ms(k-3/2)})\quad\text{as}\ s\to\infty.
\end{equation}
Since $F_k(X_1,\dots,X_k)$ is absolutely irreducible, we also have
\begin{equation}\label{Fk}
|V_{\f_{2^{ms}}^k}(F_k)\setminus V_{\f_{2^{ms}}^k}(\Delta)|=2^{ms(k-1)}+O(2^{ms(k-3/2)})\quad\text{as}\ s\to\infty.
\end{equation}
Combining \eqref{VFk}, \eqref{theta} and \eqref{Fk} gives
\[
2^{ms(k-1)}+O(2^{ms(k-3/2)})=l\,2^{ms(k-1)}+O(2^{ms(k-3/2)})\quad\text{as}\ s\to\infty.
\]
Therefore $l=1$.
\end{proof}

\section{New Results on Carlet's Conjecture}

In \cite{Carlet-Hou}, the Lang-Weil bound was applied to the absolutely irreducible polynomial $F_k(X_1,\dots,X_k)$ ($k\ge 3$) to show that $\fin$ is not $k$th order sum-free when $n\ge 10.8-5$; see \cite[Theorem~4.2 and Remark~4.4]{Carlet-Hou}. Now that we know that $\Theta_k(X_1,\dots,X_k)$ ($k\ge 3$) is also absolutely irreducible, we may apply the same method to $\Theta_k$. Since $\deg\Theta_k$ ($=2^{k-1}$) is less than $\deg F_k$ ($=2^k-2$), the Lang-Weil bound produces better estimate on $\Theta_k$ than $F_k$.

\begin{thm}\label{T5.1}
If $k\ge 3$ and
\begin{equation}\label{T5.1-eq1}
n\ge\frac{\log_2\bigl(1+\sqrt{21}\,\bigr)}3(13k-19),
\end{equation}
then $\fin$ is not $k$th order sum-free. Condition~\eqref{T5.1-eq1} is satisfied when
\begin{equation}\label{T5.1-eq2}
n\ge 10.8k-15.7.
\end{equation}
\end{thm}

\begin{proof}
The proof is almost identical to that of \cite[Theorem~4.2]{Carlet-Hou}. We may assume $k\ge 4$ since $\fin$ is not $3$rd order sum-free for $n\ge 6$ (\cite[Corollary~8]{Carlet-CEA-2024/1007}). By Theorem~3.2, it suffices to show that $|V_{\f_{2^n}^k}(\Theta_k)\setminus V_{\f_{2^n}^k}(\Delta)|>0$.

Let $q=2^n$. Since $\Theta_k$ is absolutely irreducible, by \cite[Theorem~5.2]{Cafure-Matera-FFA-2006},
\begin{align*}
|V_{\f_{2^n}^k}(\Theta_k)|\,&\ge q^{k-1}-(2^{k-1})(2^{k-1}-2)q^{k-3/2}-5(2^{k-1})^{13/3}q^{k-2}\cr
&>q^{k-1}-2^{2(k-1)}q^{k-3/2}-5\cdot 2^{13(k-1)/3}q^{k-2}.
\end{align*}
On the other hand, by \cite[Lemma~2.2]{Cafure-Matera-FFA-2006},
\[
|V_{\f_{2^n}^k}(\Theta_k)\cap V_{\f_{2^n}^k}(\Delta)|\le(2^k-1)^2q^{k-2}<2^{2k}q^{k-2}.
\]
Hence
\begin{align*}
|V_{\f_{2^n}^k}(\Theta_k)\setminus V_{\f_{2^n}^k}(\Delta)|\,&>q^{k-1}-2^{2(k-1)}q^{k-3/2}-(5\cdot 2^{13(k-1)/3}+2^{2k})q^{k-2}\cr
&=q^{k-2}\bigl(y^2-2^{2(k-1)}y-(5\cdot 2^{13(k-1)/3}+2^{2k})\bigr),
\end{align*}
where $y=q^{1/2}=2^{n/2}$. Let $y_0$ denote the larger root of the quadratic $Y^2-2^{2(k-1)}Y-(5\cdot 2^{13(k-1)/3}+2^{2k})$. We have
\begin{align*}
y_0\,&=\frac 12\bigl(2^{2(k-1)}+\sqrt{2^{4(k-1)}+20\cdot 2^{13(k-1)/3}+2^{2k+2}}\bigr)\cr
&\le\frac 12\bigl(2^{2(k-1)}+\sqrt{21\cdot 2^{13(k-1)/3}}\bigr)\kern5em\text{(since $k\ge 4$)}\cr
&\le\frac 12\bigl(1+\sqrt{21}\,\bigr)2^{13(k-1)/6}.
\end{align*}
Therefore, it suffices to show that
\[
y=2^{n/2}\ge\bigl(1+\sqrt{21}\bigr)2^{13(k-1)/6-1},
\]
i.e.,
\[
n\ge\frac{\log_2\bigl(1+\sqrt{21}\,\bigr)}3(13k-19),
\]
which is given. This completes the proof.
\end{proof}

Let us focus on the case $k=5$. In this case, condition~\eqref{T5.1-eq1} becomes $n\ge 38.041$ and condition~\eqref{T5.1-eq2} becomes $n\ge 38.3$. However, through better bookkeeping in the proof of Theorem~\ref{T5.1}, we can lower the bound for $n$ significantly. With $q=2^n$, we have
\begin{align*}
&|V_{\f_q^5}(\Theta_5)|\ge q^4-15\cdot 14\, q^{5-3/2}-5\cdot 16^{13/3}q^{k-2},\cr
&|V_{\f_q^5}(\Theta_5)\cap V_{\f_q^5}(\Delta)|\le 31^2q^{k-2},\cr
&|V_{\f_q^5}(\Theta_5)\setminus V_{\f_q^5}(\Delta)|\ge q^3[q-15\cdot 14\,q^{1/2}-(5\cdot 16^{13/3}+31^2)].
\end{align*}
Solving 
\[
2^n-15\cdot 14\cdot 2^{n/2}-(5\cdot 16^{13/3}+31^2)>0
\]
gives $n>19.9894$. Therefore, we have the following lemma.

\begin{lem}\label{L5.2}
When $n\ge 20$, $\fin$ is not $5$th order sum-free.
\end{lem}

The assumption that $n\ge 20$ in Lemma~\ref{L5.2} can be easily removed.

\begin{lem}\label{L5.3}
When $n\ge 8$, $\fin$ is not $5$th order sum-free.
\end{lem}

\begin{proof}
For $n\ge 20$, Lemma~\ref{L5.2} applies. For $8\le n\le 16$ and $n=18$, the claim follows from Result~(3) in Sections 2. For $n=17$ and $19$, the claim follows from Examples~\ref{E3.3} and \ref{E3.4}.
\end{proof}

\begin{thm}\label{T5.4}
Assume $n\ge 8$. If $3\le k\le\lceil n/3\rceil+2$ or $\lfloor 2n/3\rfloor-2\le k\le n-3$, then $\fin$ is not $k$th order sum-free.
\end{thm}

\noindent{\bf Note.} The purpose of the assumption $n\ge 8$ in Theorem~\ref{T5.4} is to avoid the cases that are already settled. In fact,
\[
n\ge 8\ \Leftrightarrow\ \left\lceil\frac n3\right\rceil+2\le n-3\ \Leftrightarrow\ \left\lfloor\frac{2n}3\right\rfloor-2\ge 3.
\]

\begin{proof}[Proof of Theorem~\ref{T5.4}]
By Results~(2) and (3) in Section~2, we may assume $n\ge 17$, and we only have to consider the case $3\le k\le\lceil n/3\rceil+2$. We know that $3,4\in\mathcal K_n$ (\cite{Carlet-CEA-2024/1007, EHRZ}). By Lemma~\ref{L5.3}, we also have $5\in\mathcal K_n$. Let $l$ be the smallest integer such that $3,4,5,\dots,l\in\mathcal K_n$. It suffices to show that $l\ge\lceil n/3\rceil+2$. Assume the contrary. Then $l-2\le\lceil n/3\rceil-1<n/3$. Since $3,l-2\in\mathcal K_n$ and $3(l-2)<n$, by Result~(5) in Section~2, $3+l-2=l+1\in\mathcal K_n$, which contradicts the maximality of $l$.
\end{proof}

We already know that for $k=3,4,5$ and $n\ge k+3$, we have $k\in\mathcal K_n$. Now Theorem~\ref{T5.4} allows us to extend the range of $k$ to $3\le k\le 12$.

\begin{cor}\label{C5.5}
For $3\le k\le 12$ and $n\ge k+3$, we have $k\in\mathcal K_n$.
\end{cor}

\begin{proof}
By Result~(3) in Section~2, we may assume that $n\ge 17$ and $\gcd(n,2\cdot 3\cdot 5)=1$.

\medskip
$3\le k\le 8$: Theorem~\ref{T5.4} applies since $\lceil n/3\rceil+2\ge 8$.

\medskip
$k=9$: When $n>18$, Theorem~\ref{T5.4} applies since $\lceil n/3\rceil+2\ge 9$. When $n=17$, $n-9\in\mathcal K_n$, whence $9\in\mathcal K_n$ by Result~(2) in Section~2.

\medskip
$k=10$: When $n>21$, Theorem~\ref{T5.4} applies since $\lceil n/3\rceil+2\ge 10$. When $17\le n\le 19$, $n-10\in\mathcal K_n$, whence $10\in\mathcal K_n$.

\medskip
$k=11$: When $n>24$, Theorem~\ref{T5.4} applies since $\lceil n/3\rceil+2\ge 11$. When $n=23$, by \cite[Table~2]{Carlet-Hou}, $11\in\mathcal K_n$. When $17\le n\le 21$, $n-11\in\mathcal K_n$, whence $11\in\mathcal K_n$.

\medskip
$k=12$: When $n>27$, Theorem~\ref{T5.4} applies since $\lceil n/3\rceil+2\ge 12$. When $17\le n\le 23$, $n-12\in\mathcal K_n$, whence $12\in\mathcal K_n$.
\end{proof}

\begin{cor}\label{C5.6}
Conjecture~\ref{C1.1} is true for odd $3\le n\le 27$.
\end{cor}

\begin{proof}
By Result~(3) in Section~2, we only have to consider $n=17,19,23$. Since $\lfloor n/2\rfloor\le 11$, by Corollary~5.5, $3,4,\dots,\lfloor n/2\rfloor\in\mathcal K_n$. By Result~(2) in Section~2, Conjecture~\ref{C1.1} is true for these values of $n$.
\end{proof}

\begin{thm}\label{T5.7}
Conjecture~\ref{C1.1} is true when $7\mid n$ and $(n,k)\ne(49,23),(49,26)$.
\end{thm}

\begin{proof}
By Result (3) in Section 2, we may assume that $\gcd(n,2\cdot 3\cdot 5)=1$ and $n<2(7-1)(7+2)-1=107$. Therefore, $n=49,77$, or $91$.

\medskip
{\bf Case 1.} $n=91=7\cdot 13$. It suffices to show that $3,\dots,45\in\mathcal K_n$. By Theorem~\ref{T5.4}, $3,\dots,33\in\mathcal K_n$. By \cite[Corollary~3.4]{EHRZ}, with $l=7$ and $r=3,\dots,9\in\mathcal K_n$, we have 
\[
\{r+7t:3\le r\le 9,\ 0\le t\le 13-r\}\subset\mathcal K_n.
\]
The set on the left contains $34,\dots,43, 45,47$, and we have $44=n-47\in\mathcal K_n$.

\medskip
{\bf Case 2.} $n=77=7\cdot 11$. It suffices to show that $3,\dots,38\in\mathcal K_n$. By Theorem~\ref{T5.4}, $3,\dots,28\in\mathcal K_n$. By \cite[Corollary~3.4]{EHRZ}, with $l=7$ and $r=3,\dots,9\in\mathcal K_n$, we have 
\[
\{r+7t:3\le r\le 9,\ 0\le t\le 11-r\}\subset\mathcal K_n.
\]
The set on the left contains $29,31,\dots,35,38,40,41,47$, and we have $30=n-47\in\mathcal K_n$, $36=n-41\in\mathcal K_n$ and $37=n-40\in\mathcal K_n$.

\medskip
{\bf Case 3.} $n=49=7\cdot 7$ and $k\ne 23,26$. It suffices to show that $3,\dots,22,24\in\mathcal K_n$. By Theorem~\ref{T5.4}, $3,\dots,19\in\mathcal K_n$. By \cite[Corollary~3.4]{EHRZ}, with $l=7$, $r=3$ and $t=3$, we have $24=3+7\cdot 3\in\mathcal K_n$. By Result (1) in Section 2, $21\in\mathcal K_n$. Therefore, it remains to show that $20,22\in\mathcal K_n$.

Since $3\in\mathcal K_7$, there is a $3$-dimensional zero-sum subspace $E\subset \f_{2^7}$. By the proof of \cite[Theorem~4.11]{EHRZ}, there is a $3$-dimensional subspace $V\subset \f_{2^{49}}$ such that $F:=E\oplus V$ is a $6$-dimensional zero-sum subspace of $\f_{2^{49}}$ . Note that $\dim_{\f_{2^7}}\f_{2^7}F\le 4$. By \cite[Theorem~3.3]{EHRZ}, with $l=7$, $r=6$ and $t=2,3$, we have $6+2\cdot 7=20\in\mathcal K_n$ and $6+3\cdot 7=27\in\mathcal K_n$. It follows that $22=n-27\in\mathcal K_n$.
\end{proof}

\section*{Acknowledgments}

The authors thank Claude Carlet for his valuable comments. 



\end{document}